\documentclass[a4paper,11pt]{amsart}

\usepackage{amsmath,amssymb}
\usepackage[T1]{fontenc}
\usepackage{latexsym}
\usepackage{enumerate}
\usepackage{fancyhdr}
\usepackage{scrpage}
\usepackage{latexsym}
\usepackage{dsfont}				
\usepackage{mathrsfs}			
\usepackage[active]{srcltx}		
\usepackage{enumitem}
\usepackage{esint}

\DeclareMathOperator{\dist}{dist}

\DeclareMathOperator{\loc}{loc}
\numberwithin{equation}{section}

\newtheorem{Def}{Definition}[section]
\newtheorem{Lemma}[Def]{Lemma}

\newtheorem{Theorem}[Def]{Theorem}
\newtheorem{Prop}[Def]{Proposition}

\newtheorem{Remark}[Def]{Remark}

\newcommand{\R}{\ensuremath{\mathbb{R}}}
\newcommand{\C}{\ensuremath{\mathbb{C}}}
\newcommand{\N}{\ensuremath{\mathbb{N}}}

\newcommand{\Eins}{\ensuremath{\mathds{1}}}

\newcommand{\Pb}{\ensuremath{\mathbb{P}}}

\newcommand{\calS}{\ensuremath{\mathcal{S}}}

\newcommand{\calE}{\ensuremath{\mathcal{E}}}
\newcommand{\calM}{\ensuremath{\mathcal{M}}}

\newcommand{\calR}{\ensuremath{\mathcal{R}}}

\newcommand{\calA}{\ensuremath{\mathcal{A}}}

\newcommand{\calT}{\ensuremath{\mathcal{T}}}

\newcommand{\norm}[1]{\left\|#1\right\|}
\newcommand{\abs}[1]{\left|#1\right|}

\newcommand{\nnorm}[1]{\|#1\|}

\let\div=\relax
\DeclareMathOperator{\div}{div}
\DeclareMathOperator{\el}{ell}
\DeclareMathOperator{\uloc}{uloc}

\newcommand{\ldot}{\,.\,}

\begin{document}

\title[A new proof for the well-posedness of NSE in $BMO^{-1}$]{A new proof for Koch and Tataru's result on the well-posedness of Navier-Stokes equations in $BMO^{-1}$}

\date{\today}

\author{Pascal Auscher}
\address{Pascal Auscher - Univ. Paris-Sud, laboratoire de Math\'ematiques, UMR 8628 du CNRS, F-91405 {\sc Orsay}} 
\email{pascal.auscher@math.u-psud.fr}

\author{Dorothee Frey}
\address{Dorothee Frey - Mathematical Sciences Institute, John Dedman Building, Australian National University, Canberra ACT 0200, Australia}
\email{dorothee.frey@anu.edu.au}

\subjclass[2010]{
35Q10,  
76D05, 
42B37, 
42B35} 

\keywords{Navier-Stokes equations; tent spaces; maximal regularity; 
Hardy spaces.}

\date{\today}

\begin{abstract} 
We give a new proof of a well-known result of Koch and Tataru on the well-posedness of Navier-Stokes equations in $\R^n$ with small initial data in $BMO^{-1}(\R^n)$. 
The proof is formulated operator theoretically and does not make use of self-adjointness of the Laplacian. 
\end{abstract}

\maketitle

\section{Introduction}

In \cite{KochTataru}, it was shown that the incompressible Navier-Stokes equations in $\R^n$ are well-posed for small initial data in $BMO^{-1}(\R^n)$. 
The result was a breakthrough, and it is believed to be  best possible, in the sense that $BMO^{-1}(\R^n)$ is the largest possible space with the scaling of $L^n(\R^n)$ where the incompressible Navier-Stokes equations are proved to be well-posed. Ill-posedness  is shown  in the largest possible space $B^{-1}_{{\infty,\infty}}(\R^n)$ in \cite{BourgainPavlovic}, and in a space between $BMO^{-1}(\R^n)$ and $B^{-1}_{{\infty,\infty}}(\R^n)$ in \cite{Yoneda}. See also some counter-examples of this type in \cite{AT}.  \\

The proof in  \cite{KochTataru} reduces to establishing the boundedness of a bilinear operator. This  proof has two main ingredients: bounds coming from the representation of the Laplacian (such as the estimates for the Oseen kernel) and, in the crucial step,   self-adjointness of the Laplacian. \\

Our new proof is rather based on operator theoretical arguments with particular emphasis on use of tent spaces, maximal regularity operators and Hardy spaces. 
 In particular, we do not make use of self-adjointness of the Laplacian. Let us mention that  our techniques  have generalisations to rougher operators, thanks to recent work on maximal regularity on tent spaces (cf. \cite{AMP} and \cite{AKMP}) and Hardy spaces associated with (bi-)sectorial operators (cf. \cite{ADM}, \cite{AMR}, \cite{HM}, \cite{HMMc} and followers). 
Using those results, it is possible to adapt our new proof to operators whose semigroup only satisfies bounds of non-pointwise type. 
This opens up the way to possible generalisations for Navier-Stokes equations on rougher domains and in other type of geometry (cf. \cite{Taylor}, \cite{MitreaTaylor}, \cite{MitreaMonniaux1}, \cite{MitreaMonniaux2} for Lipschitz domains in Riemannian manifolds, and \cite{BahouriGallagher} on the Heisenberg group), geometric flows (cf. \cite{KochLamm}),
or  other semilinear parabolic equations of a similar structure, but for rougher domains or operators (cf. \cite{MarchandLemarie} for dissipative quasi-geostrophic equations, and \cite{Giga}, \cite{HaakKunstmann} for abstract formulations of parabolic equations with quadratic nonlinearity). \\
Let us further mention potential applications to stochastic Navier-Stokes equations (cf. e.g. \cite{LiuRoeckner} and the references therein). The maximal regularity operators on tent spaces we are relying on in our proof, have proven useful already for other stochastic differential equations (cf. \cite{AvNP}).\\

Consider the incompressible Navier-Stokes equations in $\R^n \times \R^+$
\begin{align*}\tag{NSE} \label{NSE}
\left\{ \begin{array}{rl}
u_t + (u \cdot \nabla) u - \Delta u +\nabla p  \!\! &= \; 0 \\
\div u \!\! &= \;0 \\
u(0,\,.\,) \!\! &= \;u_0,
 \end{array} \right.
\end{align*}
where $u$ is the velocity and $p$ the pressure. As usual, the pressure term can be eliminated by applying the Leray projection $\Pb$.  It is known from \cite{FLRT} that the differential Navier-Stokes equations are equivalent to their integrated counterpart 
\begin{align*} 
 \left\{ \begin{array}{rl} 
 u(t,\ldot) \!\! &= \; e^{t\Delta}u_0 - \int_0^t e^{(t-s)\Delta} \Pb \div (u(s,\ldot) \otimes u(s,\ldot)) \,ds \\
 \div u_{0} \!\! &= \; 0
 \end{array} \right.
\end{align*}
under an assumption of uniform  local square integrability of $u$. (In fact, under such a control on $u$, most possible formulations of the Navier-Stokes equations are equivalent, as shown by the nice note of Dubois \cite{Dubois}.) 
Using the Picard contraction principle, matters reduce to showing that the bilinear operator $B$, defined by
\begin{align} \label{Def-BilinearOp}
	B(u,v)(t,\ldot):=\int_0^t e^{(t-s)\Delta} \Pb \div ((u \otimes v(s,\ldot))\,ds,
\end{align}
is bounded on an appropriately defined admissible path space to which the free evolution $e^{t\Delta} u_{0}$ belongs. This is what we reprove with an argument  based on boundedness of singular integrals like operators on parabolically scaled tent spaces.

\section{New Proof}
\label{section:newProof}

We use the following tent spaces.

\begin{Def}
The tent space $T^{1,2}(\R^{n+1}_+)$ is defined as the space of all functions $F \in L^2_{\loc}(\R^{n+1}_+)$ such that 
\[	
	\norm{F}_{T^{1,2}(\R^{n+1}_+)} = \int_{\R^n} \left(\iint_{R^{n+1}_+} t^{-n/2}\Eins_{B(x,\sqrt{t})}(y) \abs{F(t,y)}^2 \,dydt \right)^{1/2} dx < \infty.
\]
The tent spaces $T^{\infty,1}(\R^{n+1}_+)$ and $T^{\infty,2}(\R^{n+1}_+)$ are defined as the spaces of all functions $F \in L^2_{\loc}(\R^{n+1}_+)$ such that 
\[
	\norm{F}_{T^{\infty,p}(\R^{n+1}_+)}
	 = \sup_{x \in \R^n} \sup_{t>0} \left(t^{-n/2} \int_0^t \int_{B(x,\sqrt{t})} \abs{F(s,y)}^p \,dyds\right)^{1/p} <\infty, 
\]
for $p \in {1,2}$, respectively.\\ 
The tent space $T^{1,\infty}(\R^{n+1}_+)$ is defined as the space of all continuous functions $F:\R^{n+1}_+ \to \C$ such that the parabolic non-tangential limit $\lim_{\substack{ (t,y) \to x\\  x\in B(y,\sqrt{t}) }} F(t,y)$ exists for a.e. $x \in \R^n$ and 
\[
	\norm{F}_{T^{1,\infty}(\R^{n+1}_+)} = \nnorm{N(F)}_{L^1(\R^n)}  <\infty,\]
where $N$, defined by $N(F)(x):=\sup_{(t,y);   x\in B(y,\sqrt{t})} \abs{F(t,y)}$, denotes the non-tangential maximal function.
\end{Def}

The tent spaces were introduced in \cite{CMS}, but in elliptic scaling. 
It is easy to check that 
\[
	F \in T^{1,2}(\R^{n+1}_+) \quad \Leftrightarrow \quad 
	G \in T^{1,2}_{\el}(\R^{n+1}_+), \qquad \text{where} \ G(t,\,.\,) :=tF(t^2,\,.\,),
\]
and $T^{1,2}_{\el}(\R^{n+1}_+)$ denotes the tent space in elliptic scaling as in \cite{CMS}.
The corresponding rescaling holds true for $T^{\infty,1}(\R^{n+1}_+)$ and $T^{\infty,2}(\R^{n+1}_+)$. For $T^{1,\infty}(\R^{n+1}_+)$, however, one has $G(t,\,.\,):=F(t^2,\,.\,)$ instead.\\
One has the duality $(T^{1,2}(\R^{n+1}_+))' = T^{\infty,2}(\R^{n+1}_+)$ and $(T^{1,\infty}(\R^{n+1}_+))' = T^{\infty,1}(\R^{n+1}_+)$ for the $L^2$ inner product  $\iint_{\R^{n+1}_{+}} f(t,y) \overline {g(t,y)}\, dydt$.\\

We recall the definition of the admissible path space for \eqref{NSE}  in \cite{KochTataru} (with the notation as in \cite{Lemarie}).

\begin{Def}
Let $T \in (0,\infty]$. Define
\[
	\calE_T:=\{ u \in L^2_{\uloc,x}L^2_t ((0,T) \times \R^n) \,:\, \norm{u}_{\calE_T} <\infty\},
\]
with
\[
	 \norm{u}_{\calE_T} := \sup_{0<t<T} \norm{t^{1/2}u(t,\,.\,)}_{L^{\infty}(\R^n)} + \sup_{x \in \R^n} \sup_{0<t<T} \left(t^{-n/2} \int_0^t \int_{B(x,\sqrt{t})} \abs{u(s,y)}^2 \,dyds\right)^{1/2}.
\]
\end{Def}

\begin{Remark}
(i) Observe that for $T=\infty$, one has 
\begin{align} \label{Def-ET-infinity}
	\norm{u}_{\calE_\infty} 
	= \sup_{t>0} \norm{t^{1/2} u(t,\,.\,)}_{L^\infty(\R^n)} 
	+ \norm{u}_{T^{\infty,2}(\R^{n+1}_+)}.
\end{align}
(ii) The corresponding adapted value space $E_T$ is defined as the space for which $u_0 \in E_T$ if and only if $u_0 \in \calS'(\R^n)$ and $(e^{t\Delta}u_0)_{0<t<T} \in \calE_T$. 
Observe that the first part of the norm in \eqref{Def-ET-infinity} corresponds to the adapted value space $\dot{B}^{-1}_{\infty,\infty}(\R^n)$ and the second part to $BMO^{-1}(\R^n)$. Since $BMO^{-1}(\R^n) \hookrightarrow \dot{B}^{-1}_{\infty,\infty}(\R^n)$, one has $E_\infty=BMO^{-1}(\R^n)$.
\end{Remark}

\begin{Theorem} \label{mainTheorem}
Let $T \in (0,\infty]$. The bilinear operator $B$ defined in \eqref{Def-BilinearOp} is continuous from $(\calE_T)^n \times (\calE_T)^n$ to $(\calE_T)^n$.
\end{Theorem}

\begin{proof}
We restrict ourselves to the case $T=\infty$. The same argument works otherwise. \\
\textbf{Step 1} (From linear to bilinear).
In a first step, one reduces the bilinear estimate to a linear estimate. We use the following fact, which is a simple consequence of H\"older's inequality:
\begin{align} \label{alpha-condition}
u,v \in (\calE_\infty)^n, \; \alpha:=u \otimes v 
\quad \Rightarrow \quad
\begin{cases} 
 \alpha \in T^{\infty,1}(\R^{n+1}_+;\C^n \otimes \C^n), \\
 s^{1/2}\alpha(s,\,.\,) \in T^{\infty,2}(\R^{n+1}_+;\C^n \otimes \C^n), \\
 s \alpha(s,\,.\,) \in L^{\infty}(\R^{n+1}_+;\C^n \otimes \C^n).
 \end{cases}
\end{align}
It thus suffices to show that for the operator $\calA$ defined by
\[
	\calA(\alpha)(t,\ldot) = \int_0^t e^{(t-s)\Delta} \Pb \div \alpha(s,\,.\,) \, ds,
\]
there exists a constant $C>0$ such that for all $\alpha$ satisfying the conditions in \eqref{alpha-condition},

\begin{align} \label{linear-est1}
	\norm{t^{1/2}\calA(\alpha)}_{L^{\infty}(\R^{n+1}_+;\C^n)} 
	& \leq C \norm{\alpha}_{T^{\infty,1}(\R^{n+1}_+;\C^n \otimes \C^n)} + \norm{s\alpha(s,\,.\,)}_{L^{\infty}(\R^{n+1}_+;\C^n \otimes \C^n)}, \\ 
	\label{linear-est2}
	\norm{\calA(\alpha)}_{T^{\infty,2}(\R^{n+1}_+;\C^n)}
	 &\leq C \norm{\alpha}_{T^{\infty,1}(\R^{n+1}_+;\C^n \otimes \C^n)}
	 + \norm{s^{1/2}\alpha(s,\,.\,)}_{T^{\infty,2}(\R^{n+1}_+;\C^n \otimes \C^n)}.
\end{align}
\textbf{Step 2} ($L^\infty$ estimate).\\
We omit the proof of \eqref{linear-est1}, noticing that the proof of \eqref{linear-est1} in \cite{KochTataru} only uses the polynomial bounds  on  the Oseen kernel $k_t(x)$ of $e^{t\Delta}\Pb$ (See e.g. \cite{Lemarie}, Chapter 11) for $|\beta|=1$, 
\[
	\abs{ t^{\abs{\beta}/2}\partial_\beta k_t(x) } \leq C t^{-n/2}(1+t^{-1/2}\abs{x})^{-n-\abs{\beta}} 
	\qquad \forall \beta \in \N^n, \; \forall x \in \R^n,
\]
and no other special properties on the corresponding operator $e^{t\Delta}\Pb$.\\

\textbf{Step 3} ($T^{\infty,2}$ estimate - New decomposition).\\
We split $\calA$ into three parts:
\begin{align*}
 \calA(\alpha)(t,\,.\,) &= \int_0^t e^{(t-s)\Delta} \Pb \div \alpha(s) \,ds \\
  & =\int_0^t e^{(t-s)\Delta} \Delta  (s(-\Delta))^{-1} (I-e^{2s\Delta}) s^{1/2} \Pb \div s^{1/2} \alpha(s,\,.\,)\,ds \\
  & \qquad + \int_0^{\infty} e^{(t+s)\Delta} \Pb \div \alpha(s,\,.\,)\,ds\\
  & \qquad -  \int_t^\infty e^{(t+s)\Delta} \Pb s^{-1/2} \div s^{1/2}\alpha(s,\,.\,)\,ds\\
  & =:\calA_1(\alpha)(t,\ldot)+ \calA_2(\alpha)(t,\ldot) + \calA_3(\alpha)(t,\ldot).
\end{align*}

\textbf{Step 3(i)} (Maximal regularity operator).
To treat  $\calA_1$, we use the fact that the maximal regularity operator 
\begin{align} \label{def-maxreg} \nonumber
 & \calM^+ : 	T^{\infty,2}(\R^{n+1}_+) \to T^{\infty,2}(\R^{n+1}_+), \\
  & (\calM^+ F)(t,\,.\,) := \int_0^t e^{(t-s)\Delta} \Delta F(s,\,.\,) \,ds,
\end{align}
is bounded. 
The result for $T^{2,2}(\R^{n+1}_+) = L^2(\R^{n+1}_+)$ was established by de Simon in \cite{deSimon}.
The extension to $T^{\infty,2}(\R^{n+1}_+)$ was implicit in [9], but not formulated this way. It  is an application of \cite{AMP}, Theorem 3.2, taking $\beta=0$, $m=2$ and $L=-\Delta$, noting that the Gaussian bounds for the kernel of $t\Delta e^{t\Delta}$ yield the needed decay.\\
Next, for $s>0$, define $T_s:= (s(-\Delta))^{-1}(I-e^{2s\Delta}) s^{1/2}\Pb \div$. Observe that $T_s$ is bounded uniformly in $L^2(\R^n)$, and that $T_s$ is an integral operator with convolution kernel $k_s$ satisfying estimates of order $n+1$ at $\infty$, \textit{i.e.} 
\begin{align} \label{kernel-est-Ts}
	\abs{k_s(x)} \leq C s^{-n/2}(s^{-1/2}\abs{x})^{-n-1} \qquad \forall x \in \R^n, \, \abs{x} \geq s^{1/2}.
\end{align}
We show in Lemma \ref{boundedness-calT} below that the operator $\calT$, defined by
\begin{align} \label{def-multop} \nonumber
	& \calT: T^{\infty,2}(\R^{n+1}_+;\C^n \otimes \C^n) \to T^{\infty,2}(\R^{n+1}_+;\C^n), \\
	& (\calT F)(s,\,.\,) := T_s (F(s,\,.\,)),
\end{align}
is bounded.
With the definitions in \eqref{def-maxreg} and \eqref{def-multop}, we then have $\calA_1(\alpha) = \calM^+\calT (s^{1/2}\alpha(s,\,.\,))$ and the boundedness of these operators imply
\begin{align*}
 \norm{\calA_1(\alpha)}_{T^{\infty,2}}
 &= \norm{\calM^+\calT (s^{1/2}\alpha(s,\,.\,))}_{T^{\infty,2}} \\
 	&\lesssim \norm{\calT (s^{1/2}\alpha(s,\,.\,))}_{T^{\infty,2}} \lesssim \norm{s^{1/2}\alpha(s,\,.\,)}_{T^{\infty,2}}.
\end{align*}

\textbf{Step 3(ii)} (Hardy space estimates).
This is the main new part of the proof.  We use in the following that the Leray projection $\Pb$ commutes with the Laplacian 
and the above bounds on the Oseen kernel to show that 
\begin{align} \label{def-singularint} \nonumber
 &\calA_{2} : T^{\infty,1}(\R^{n+1}_+;\C^n \otimes \C^n) \to T^{\infty,2}(\R^{n+1}_+;\C^n),\\
 &(\calA_{2} F)(t,\,.\,) := \int_0^\infty e^{(t+s)\Delta}\Pb\div F(s,\,.\,) \,ds,
\end{align}
is bounded. 
 We work via dualisation and
 it is enough to show  that
 \begin{align} \label{def-A2-dual} \nonumber
 &\calA_{2}^\ast : T^{1,2}(\R^{n+1}_+;\C^n) \to T^{1,\infty}(\R^{n+1}_+;\C^n \otimes \C^n),\\
  &(\calA_{2}^\ast G)(s,\,.\,) = e^{s\Delta}  \int_0^{\infty} \nabla \Pb e^{t\Delta} G(t,\,.\,) \,dt,
\end{align}
is bounded. 
To see this, we factor $\calA_{2}^\ast$ through the Hardy space $H^1(\R^n; \C^n \otimes \C^n)$.
We know from classical Hardy space theory, that $H^1(\R^n)$ can either be defined via non-tangential maximal functions or via square functions (here in parabolic scaling instead of the more commonly used elliptic scaling). First,  the operator 
\begin{align*} 
& \calS : T^{1,2}(\R^{n+1}_+;\C^n) \to H^1(\R^n; \C^n \otimes \C^n),\\
 & \calS G(\,.\,) =  \int_0^{\infty} \nabla \Pb e^{t\Delta} G(t,\,.\,) \,dt,
\end{align*}
is bounded. This uses the polynomial decay of order $n+1$ at $\infty$ of the kernel of $\nabla \Pb e^{t\Delta}$ (some weaker decay of non-pointwise type would suffice for this,  in fact). The precise calculations are given in \cite{FS} (cf. also \cite{CMS}).  Second,   again by \cite{FS},  we have for $h \in H^1(\R^n)$ that $(s,x) \mapsto e^{s\Delta} \,h(x) \in T^{1,\infty}$ and  $\norm{N(e^{s\Delta} \, h)}_{L^1(\R^n)} \lesssim \norm{h}_{H^1(\R^n)}$. The same holds componentwise for $\C^n \otimes \C^n$ valued functions. A combination of both estimates gives the result.\\

\textbf{Step 3(iii)} (Remainder term). The considered integral in $\calA_3$ is not singular in $s$ and  is an error term.    It suffices to show that
\begin{align} \label{def-errorterm} \nonumber
  &\calR : T^{\infty,2}(\R^{n+1}_+;\C^n \otimes \C^n) \to T^{\infty,2}(\R^{n+1}_+;\C^n),\\
  &(\calR F)(t,\,.\,):= \int_t^{\infty} e^{(t+s)\Delta} \Pb s^{-1/2} \div F(s,\,.\,) \,ds
\end{align}
is bounded as $\calA_{3}(\alpha)= \calR (s^{1/2} \alpha(s,\,.\,))$. This can be seen as a special case of \cite{AKMP}, Theorem 4.1 (2). We give a self-contained proof in Lemma \ref{Lemma-errorterm} below.
\end{proof}

\section{Technical results}

\begin{Lemma} \label{boundedness-calT}
Let $(T_s)_{s>0}$ be a family of uniformly bounded operators in $L^2(\R^n)$, which satisfy $L^2$-$L^\infty$ off-diagonal estimates of the form 
\begin{align} \label{L2-Linfty-Ts}
 	\norm{\Eins_E T_s \Eins_{\tilde E}}_{L^2(\R^n) \to L^{\infty}(\R^n)}
 		\leq C s^{-\frac{n}{4}} \left(s^{-1/2}\dist(E,\tilde E)\right)^{-\frac{n}{2}-1}
\end{align}
for all Borel sets $E,\tilde E \subseteq \R^n$ with $\dist(E,\tilde E) \geq s^{1/2}$. Then the operator $\calT$, defined by
\begin{align*} 
	& \calT: T^{\infty,2}(\R^{n+1}_+) \to T^{\infty,2}(\R^{n+1}_+), \\
	& (\calT F)(s,\,.\,) := T_s (F(s,\,.\,)),
\end{align*}
is bounded.
\end{Lemma}

\begin{Remark}
A straightforward calculation shows that the kernel estimates in \eqref{kernel-est-Ts} imply the $L^2$-$L^\infty$ off-diagonal estimates in \eqref{L2-Linfty-Ts}.
\end{Remark}

\begin{proof}
The proof is a slight modification of \cite{HvNP}, Theorem 5.2.
Let $F \in T^{\infty,2}(\R^{n+1}_+)$ and fix $(t,x) \in \R^{n+1}_+$. 
Define $F_0:=\Eins_{B(x,2\sqrt{t})}F$ and $F_j:=\Eins_{B(x,2^{j+1}\sqrt{t})\setminus B(x,2^j\sqrt{t})}F$ for $j \geq 1$. The uniform boundedness of $T_s$ in $L^2(\R^n)$ yields 
\[
\norm{T_sF_0(s,\,.\,)}_{L^2(B(x,\sqrt{t}))} \lesssim \norm{F(s,\,.\,)}_{L^2(B(x,2\sqrt{t}))}.
\]
For $s<t$ and $j \geq 1$, on the other hand,
H\"older's inequality and \eqref{L2-Linfty-Ts} yield
\begin{align*}
	& \norm{T_sF_j(s,\,.\,)}_{L^2(B(x,\sqrt{t}))}
		\lesssim t^{\frac{n}{4}} \norm{T_s F_j(s,\,.\,)}_{L^{\infty}(B(x,\sqrt{t}))}\\
		& \qquad \lesssim t^{\frac{n}{4}} s^{-\frac{n}{4}} \left(\frac{\sqrt{s}}{2^j\sqrt{t}}\right)^{\frac{n}{2}+1} \norm{F(s,\,.\,)}_{L^2(B(x,2^{j+1}\sqrt{t}))}
		\lesssim 2^{-j(\frac{n}{2}+1)} \norm{F(s,\,.\,)}_{L^2(B(x,2^{j+1}\sqrt{t}))}.
\end{align*}
Thus, 
\begin{align*}
	& \left(t^{-n/2} \int_0^t \norm{T_sF(s,\,.\,)}_{L^2(B(x,\sqrt{t}))}^2 \,ds\right)^{1/2} \\
	& \qquad \lesssim \sum_{j \geq 0} 2^{-j(\frac{n}{2}+1)} 2^{j\frac{n}{2}} \left((2^j\sqrt{t})^{-n} \int_0^t \norm{F(s,\,.\,)}_{L^2(B(x,2^{j+1}\sqrt{t}))} \,ds\right)^{1/2} 
	\lesssim \norm{F}_{T^{\infty,2}(\R^{n+1}_+)}.
\end{align*}
\end{proof}

\begin{Lemma} \label{Lemma-errorterm}
The operator $\calR$ defined in \eqref{def-errorterm} is bounded.
\end{Lemma}

\begin{proof}
We write
\[
	(\calR F)(t,\,.\,) = \int_t^\infty K(t,s)F(s,\,.\,)\,ds,
\] 
with $K(t,s):=e^{(t+s)\Delta} \Pb s^{-1/2} \div$ for $s,t>0$. We first show the boundedness of $\calR$ on $L^2(\R^{n+1}_+)$.
This follows from the easy bound   $\norm{K(t,s)}_{L^2(\R^n) \to L^2(\R^n)} \leq C s^{-1/2}(t+s)^{-1/2}$. Indeed, pick some $\beta \in (-\frac{1}{2},0)$, set $p(t):=t^{\beta}$ and observe that  
$k(t,s) := \Eins_{(t,\infty)}(s) \norm{K(t,s)}_{L^2(\R^n) \to L^2(\R^n)}$ satisfies
\begin{align*}
 \int_0^{\infty} {k(t,s)} p(t) \,dt &\lesssim  \int_0^s s^{-1/2} t^{-1/2} t^{\beta} \,dt \lesssim s^{\beta}=p(s),\\
  \int_0^{\infty} {k(t,s)} p(s) \,ds &\lesssim \int_t^{\infty} s^{-1/2} s^{-1/2} s^{\beta} \,ds \lesssim t^{\beta}=p(t).
\end{align*}
This allows to apply Schur's lemma.

Next, we show that $\calR$ extends to a bounded operator on $T^{\infty,2}(\R^{n+1}_+)$. Note that for all $s,t>0$, the operator $K(t,s)$ is an  integral operator of convolution with  $k_{t,s}$, which satisfies
\begin{align} \label{kernel-kts}
	\abs{k_{t,s}(x)} \leq C s^{-1/2} (t+s)^{-1/2} (t+s)^{-\frac{n}{2}}\left(1+(t+s)^{-1/2}\abs{x}\right)^{-n-1} \qquad \forall x \in \R^n.
\end{align}
These estimates imply $L^2$-$L^\infty$ off-diagonal estimates of the form
\begin{align} \label{L2-Linfty-est-M-}
	\norm{\Eins_E K(t,s) \Eins_{\tilde E}}_{L^2(\R^n) \to L^{\infty}(\R^n)}
	\leq C s^{-1/2} (t+s)^{-1/2} (t+s)^{-\frac{n}{4}}\left(1+(t+s)^{-1/2}\dist(E,\tilde E)\right)^{-\frac{n}{2}-1}
\end{align}
for all Borel sets $E,\tilde E \subseteq \R^n$.

Let $F \in T^{\infty,2}(\R^{n+1}_+)$ and fix $(r,x_0) \in \R^{n+1}_+$. Define
$B_j:=(0,2^j r) \times B(x_0,2^j \sqrt{r})$ for $j \geq 0$ and $C_j:=B_j \setminus B_{j-1}$ for $j \geq 1$.
Then set $F_0:=\Eins_{B_0} F$ and $F_{j}:=\Eins_{C_j}F$ for $j \geq 1$. Using Minkowski's inequality, we have
\begin{align*}
	&\left(r^{-n/2}\int_0^r \norm{(\calR F)(t,\,.\,)}_{L^2(B(x_0,\sqrt{r}))}^2 \,dt\right)^{1/2}\\ 
	& \qquad \lesssim \sum_{j \geq 0} \left(r^{-n/2}\int_0^r \norm{(\calR F_j)(t,\,.\,)}_{L^2(B(x_0,\sqrt{r}))}^2 \,dt\right)^{1/2} =:\sum_{j \geq 0} I_j.
\end{align*}
For $j=0$, the boundedness of $\calR$ on $L^2(\R^{n+1}_+)$ yields the desired estimate $\norm{I_0} \lesssim \norm{F}_{T^{\infty,2}(\R^{n+1}_+)}$. 
For $j \geq 1$, we split the integral in $s$ and use H\"older's inequality to obtain
\begin{align} \label{est-Ij}
	I_j \lesssim \sum_{k \geq 0} \left(r^{-n/2} \int_0^r \int_{2^kt}^{2^{k+1}t} (2^k t) \norm{K(t,s)F_j(s,\,.\,)}_{L^2(B(x_0,\sqrt{r}))}^2 \,dsdt\right)^{1/2}
\end{align}
Now observe that for $j \geq 1$, $k \geq 0$, $t \in (0,r)$ and $s \in (2^kt,2^{k+1}t)$, H\"older's inequality and \eqref{L2-Linfty-est-M-} yield for $\delta \in (0,1]$
\begin{align*}
	& \norm{K(t,s) F_j(s,\,.\,)}_{L^2(B(x_0,\sqrt{r}))}
		 \lesssim r^{n/4} \norm{K(t,s) F_j(s,\,.\,)}_{L^\infty(B(x_0,\sqrt{r}))} \\
		& \qquad  \lesssim r^{n/4} s^{-1/2} (t+s)^{-1/2} (t+s)^{-n/4} \left(1+\frac{2^j\sqrt{r}}{(t+s)^{1/2}}\right)^{-\frac{n}{2}-\delta} \norm{F_j(s,\,.\,)}_{L^2} \\
		& \qquad \lesssim (2^j)^{-\frac{n}{2}-\delta} r^{-\delta/2} (2^kt)^{-1+\delta/2} \norm{F_j(s,\,.\,)}_{L^2}.
\end{align*}
Inserting this into \eqref{est-Ij}, interchanging the order of integration and choosing $\delta \in (0,1)$ finally gives
\begin{align*}
 \sum_{j \geq 1} I_j \lesssim \sum_{j \geq 1} \sum_{k \geq 0} 2^{-j\delta} 2^{-k(\frac{1}{2}-\frac{\delta}{2})}  \left((2^j \sqrt{r})^{-n}\int_0^{2^jr} \norm{F_j(s,\,.\,)}_{L^2}^2 \,ds\right)^{1/2} 
 \lesssim \norm{F}_{T^{\infty,2}(\R^{n+1}_+)}.
\end{align*}
\end{proof}

\section{Comments}

Let us denote by $T^{\infty,2}_{1/2}(\R^{n+1}_+)$ the weighted tent space defined by $F \in T^{\infty,2}_{1/2}(\R^{n+1}_+)$ if and only if $s^{1/2}F(s,\,.\,) \in T^{\infty,2}(\R^{n+1}_+)$. Respectively for $T^{2,2}_{1/2}(\R^{n+1}_+)$.\\

The first comment is that the  $T^{\infty,2}_{1/2}$ estimate for $\alpha$ is not used in \cite{KochTataru}.\\

The second comment is  that our proof is non local in time. By this, we mean that we need to know $\alpha= u\otimes v$ on the full time interval $[0,T]$ to get estimates for $B(u,v)$  at all smaller times $t$.  In contrast, the proof in \cite{KochTataru} is local in time: bounds for $u,v$ on the time interval $[0,t]$ suffices to get bounds at time $t$ for $B(u,v)$. \\

The third comment is on the optimality of the estimate in \eqref{linear-est2} which could be related to the second comment. 
We have seen in Section \ref{section:newProof} that both $\calA_1$ and $\calA_3$ are bounded operators from $T^{\infty,2}_{1/2}$ to $T^{\infty,2}$. It is thus a natural question whether the same holds for $\calA_{2}$ as it would eliminate the $T^{\infty,1}$ term in the right hand side of \eqref{linear-est2}.
We show that this is not the case. It is therefore necessary to use a different argument for $\calA_2$, as   is done in  Step 3(ii) above. In  \cite{KochTataru}, this operator does not arise.

\begin{Prop}
The operator $\calA_2$ is neither bounded as an operator from $T^{2,2}_{1/2}(\R^{n+1}_+;\C^n \otimes \C^n)$ to $T^{2,2}(\R^{n+1}_+;\C^n )$, nor from $T^{\infty,2}_{1/2}(\R^{n+1}_+;\C^n \otimes \C^n)$ to $T^{\infty,2}(\R^{n+1}_+;\C^n)$.
\end{Prop}

We adapt the argument of  \cite{AuscherAxelsson}, Theorem 1.5.

\begin{proof}
We first show the result for $T^{2,2}$. 
We work with the dual operator $\calA_2^\ast$ defined in \eqref{def-A2-dual} and show that  
 \begin{align*}
G\mapsto   s^{-1/2}(\calA_{2}^\ast  G)(s,\,.\,) = s^{-1/2} e^{s\Delta}  \int_0^{\infty} \nabla \Pb e^{t\Delta} G(t,\,.\,) \,dt
\end{align*}
is not bounded from  $L^2(\R^{n+1}_+;\C^n)=T^{2,2}(\R^{n+1}_+;\C^n)$ to $L^2(\R^{n+1}_+;\C^n\otimes \C^n)$ .\\
There exists $u \in L^2(\R^n; \C^n)$ with $\nabla (-\Delta)^{-1/2} (-\Delta)^{-1/2}\Pb (e^{\Delta}-e^{2\Delta})u \neq 0$ in  $  L^2(\R^n;\C^n\otimes \C^n)$. Define $G(t,\,.\,)=u$ for $t \in (1,2)$, and $G(t,\,.\,)=0$ otherwise. Clearly $G\in L^{2}(\R^{n+1}_+;\C^n)$.  Then, for $s<1$,
\begin{align} \label{repr-A_2} \nonumber
	s^{-1/2}(\calA_{2}^\ast  G)(s,\,.\,)
	&= e^{s\Delta} \nabla (-\Delta)^{-1/2} (s(-\Delta))^{-1/2}  \int_1^2 (-\Delta)\Pb e^{t\Delta} u \,dt \\
	&= e^{s\Delta}\nabla (-\Delta)^{-1/2} (s(-\Delta))^{-1/2} \Pb (e^{\Delta} - e^{2\Delta}) u,
\end{align}
and
\begin{align*}
	\norm{s^{-1/2}(\calA_{2}^\ast  G)(s,\,.\,)}_{L^2(\R^{n+1}_+)}^2
		\geq \int_0^1 \norm{e^{s\Delta}\nabla (-\Delta)^{-1/2} (-\Delta)^{-1/2} \Pb (e^{\Delta} - e^{2\Delta}) u}_2^2 \,\frac{ds}{s} = \infty,
\end{align*}
as $e^{s\Delta} \to I$ for $s \to 0$. \\
For the result on $T^{\infty,2}$, we argue similarly. There is some ball $B=B(x,1)$  in $\R^n$ 
such that  $\nabla (-\Delta)^{-1/2} (-\Delta)^{-1/2}\Pb (e^{\Delta}-e^{2\Delta})u \neq 0$ in $L^2(B;\C^n\otimes \C^n)$. 
Let $G$ be defined as above.  Then $G \in T^{\infty,2}(\R^{n+1}_+;\C^n)$, since the Carleson norm of $G$ can be restricted to balls of radius larger than $1$ by definition of $G$ and
\begin{align*}
	\norm{G}_{T^{\infty,2}}^2 
	= \sup_{x_0 \in \R^n} \sup_{r>1} r^{-n/2} \int_0^r \int_{B(x_0,\sqrt{r})} \abs{G(t,x)}^2 \,dxdt 
	 \leq \int_1^2 \int_{\R^n} \abs{u(x)}^2 \,dxdt = \norm{u}_2^2.
\end{align*}
Now, using again \eqref{repr-A_2}, we get as above
\begin{align*}
 	\norm{s^{-1/2}(\calA_{2}^\ast  G)(s,\,.\,)}_{T^{\infty,2}}^2
 	\geq \int_0^1 \norm{e^{s\Delta}\nabla (-\Delta)^{-1/2} (-\Delta)^{-1/2} \Pb (e^{\Delta} - e^{2\Delta}) u}_{L^2(B)}^2 \,\frac{ds}{s} = \infty.
\end{align*}
\end{proof}

\section*{Acknowledgement}
 The  first author was partially supported by the ANR project ``Harmonic analysis at its boundaries'' ANR-12-BS01-0013-01. The second author was partially supported by the Karlsruhe House of Young Scientists (KHYS) and the Australian Research Council Discovery grants DP110102488 and DP120103692.
The second author thanks the D{\'e}partement de Math{\'e}matiques d'Orsay for their kind hospitality where this research was mostly undertaken, and would like to thank Peer Chr. Kunstmann for bringing the problem to our attention and for fruitful discussions.

\small{

}


\begin{thebibliography}{99}

\bibitem{AuscherAxelsson}
P.~Auscher, A.~Axelsson.
\newblock {\em Remarks on maximal regularity.}
\newblock {{Parabolic problems. The Herbert Amann Festschrift.} Basel: Birkh\"auser. Progress in Nonlinear Differential Equations and Their Applications 80, 45-55, 2011.}

\bibitem{ADM}
P.~Auscher, X.T. Duong and A.~McIntosh.
\newblock {\em Boundedness of Banach space valued singular integral operators and
  Hardy spaces.}
\newblock {Unpublished manuscript}, 2002.

\bibitem{AKMP}
P.~Auscher, C.~Kriegler, S.~Monniaux and P.~Portal.
\newblock {\em Singular integral operators on tent spaces.}
\newblock {J. Evol. Equ.}, 12(4):741--765, 2012. 

\bibitem{AMR}
P.~Auscher, A.~McIntosh and E.~Russ.
\newblock {\em Hardy spaces of differential forms on Riemannian manifolds.}
\newblock {J. Geom. Anal.}, 18(1):192--248, 2008.

\bibitem{AMP}
P.~Auscher, S.~Monniaux and P.~Portal.
\newblock {\em The maximal regularity operator on tent spaces.}
\newblock {Commun. Pure Appl. Anal.}, 11(6):2213--2219, 2012.

\bibitem{AvNP}
P.~Auscher, J.~van Neerven and P.~Portal.
\newblock {\em Conical stochastic maximal $L^p$-regularity for $1 \leq p < \infty$.}
\newblock {Preprint}, arXiv:1112.3196 [math.CA], 2011.

\bibitem{AT}
P.~Auscher and P.~Tchamitchian.
\newblock {\em Espaces critiques pour le syst\`{e}me des \'{e}quations de Navier-Stokes incompressibles.}
\newblock {Preprint}, arXiv:0812.1158 [math.AP], 1999.

\bibitem{BahouriGallagher}
H.~Bahouri and I.~Gallagher.
\newblock {\em The heat kernel and frequency localized functions on the Heisenberg group.}
\newblock {Progr. Nonlinear Differential Equations
Appl.}, 78:17--35, Birkh\"auser, 2009.

\bibitem{BourgainPavlovic}
J.~Bourgain and N.~{Pavlovi\'c}.
\newblock {\em Ill-posedness of the Navier-Stokes equations in a critical space in 3D.}
\newblock {J. Funct. Anal.}, 255(9):2233--2247, 2008.

\bibitem{CMS}
R.R. Coifman, Y.~Meyer and E.M. Stein.
\newblock {\em Some new function spaces and their applications to harmonic
  analysis.}
\newblock {J. Funct. Anal.}, 62:304--335, 1985.

\bibitem{Dubois}
S.~Dubois.
\newblock {\em What is a solution to the Navier-Stokes equations?}
\newblock {C. R., Math., Acad. Sci. Paris}, 335(1):27--32, 2002.
  
\bibitem{FS}
C.L. Fefferman and E.M. Stein.
\newblock {\em $H^p$ spaces of several variables.}
\newblock {Acta Math.}, 129:137--193, 1972.

\bibitem{FLRT}
G.~Furioli, P.-G.~Lemari{\'e}-Rieusset and E.~Terraneo.
\newblock {\em Unicit\'e dans $L^3(\mathbb R^3)$ et d'autres espaces fonctionnels limites pour Navier-Stokes.}
\newblock {Rev. Mat. Iberoam.}, 16(3):605--667, 2000.

\bibitem{Giga}
Y.~Giga.
\newblock {\em Solutions for semilinear parabolic equations in $L\sp p$ and regularity of weak solutions of the Navier-Stokes system.}
\newblock {J. Differ. Equations}, 61:186--212, 1986.

\bibitem{HaakKunstmann}
B.H.~Haak and P.C.~Kunstmann.
\newblock {\em On Kato's method for Navier-Stokes equations.}
\newblock {J. Math. Fluid Mech.}, 11(4):492--535, 2009.

\bibitem{HM}
S.~Hofmann and S.~Mayboroda.
\newblock {\em Hardy and BMO spaces associated to divergence form elliptic
  operators.}
\newblock {Math. Ann.}, 344(1):37--116, 2009.

\bibitem{HMMc}
S.~Hofmann, S.~Mayboroda and A.~McIntosh.
\newblock {\em Second order elliptic operators with complex bounded measurable
  coefficients in $L^p$, Sobolev and Hardy spaces.}
\newblock {Ann. Sci. \'Ec. Norm. Sup\'er. (4)}, 44(5):723--800, 2011.

\bibitem{HvNP}
T.~{Hyt\"onen}, J.~van Neerven and P.~Portal.
\newblock{\em Conical square function estimates in UMD Banach spaces and applications to $H^{\infty}$-functional calculi.}
\newblock {J. Anal. Math.}, 106:317--351, 2008.

\bibitem{KochLamm}
H.~Koch and T.~Lamm.
\newblock {\em Geometric flows with rough initial data.}
\newblock {Asian J. Math.}, 16(2):209--235, 2012.

\bibitem{KochTataru}
H.~Koch and D.~Tataru.
\newblock {\em Well-posedness for the Navier-Stokes equations.}
\newblock {Adv. Math.}, 157(1):22--35, 2001. 

\bibitem{Lemarie}
P.~G.~Lemari{\'e}-Rieusset.
\newblock {\em Recent developments in the Navier-Stokes problem.}
\newblock {Chapman \& Hall/CRC Research Notes in Mathematics Series,} 2002.

\bibitem{MarchandLemarie}
P.~G.~Lemari{\'e}-Rieusset and F.~Marchand.
\newblock {\em Solutions auto-similaires non radiales pour l'\'equation quasi-g\'eostrophique dissipative critique.}
\newblock {C. R., Math., Acad. Sci. Paris}, 341(9):535--538, 2005.

\bibitem{LiuRoeckner}
W.~Liu and M.~R\"ockner.
\newblock {\em Local and global well-posedness of SPDE with generalized coercivity conditions.}
\newblock {J. Differ. Equations}, 254(2):725--755, 2013.

\bibitem{MitreaMonniaux1}
M.~Mitrea and S.~Monniaux.
\newblock {\em On the analyticity of the semigroup generated by the Stokes operator with Neumann-type boundary conditions on Lipschitz subdomains of Riemannian manifolds.}
\newblock {Trans. Am. Math. Soc.}, 361(6), 3125--3157, 2009.

\bibitem{MitreaMonniaux2}
M.~Mitrea and S.~Monniaux.
\newblock {\em The nonlinear Hodge-Navier-Stokes equations in Lipschitz domains.}
\newblock {Differ. Integral Equ.}, 22(3-4):339--356, 2009.

\bibitem{MitreaTaylor}
M.~Mitrea and M.E.~Taylor.
\newblock {\em Navier-Stokes equations on Lipschitz domains in Riemannian manifolds.}
\newblock {Math. Ann.}, 321(4):955--987, 2001.

\bibitem{deSimon}
L.~de Simon. 
\newblock {\em Un'applicazione della theoria degli integrali singolari allo studio delle equazioni differenziali lineare astratte del primo ordine.}
\newblock {Rend. Sem. Mat.}, Univ. Padova 205--223, 1964. 

\bibitem{Taylor}
M.E.~Taylor. 
\newblock {\em Incompressible fluid flows on rough domains.}
\newblock {Progr. Nonlinear Differential Equations
Appl.}, 42:320--334, Birkh\"auser, 2000.

\bibitem{Yoneda}
T.~Yoneda.
\newblock {\em Ill-posedness of the 3D-Navier-Stokes equations in a generalized Besov space near $\mathrm{BMO}^{-1}$.}
\newblock {J. Funct. Anal.}, 258(10):3376--3387, 2010.

\end{thebibliography}
\end{document}